\documentclass[11pt]{article}

\usepackage[a4paper,margin=1in]{geometry}
\usepackage{amsmath,amssymb,amsthm,mathtools}
\usepackage{microtype}
\usepackage{hyperref}

\hypersetup{
  colorlinks=true,
  linkcolor=blue,
  citecolor=red,
  urlcolor=blue
}

\numberwithin{equation}{section}

\newtheorem{theorem}{Theorem}[section]
\newtheorem{proposition}[theorem]{Proposition}
\newtheorem{lemma}[theorem]{Lemma}
\newtheorem{remark}[theorem]{Remark}

\newcommand{\R}{\mathbb{R}}
\newcommand{\dd}{\,\mathrm{d}}
\newcommand{\Div}{\operatorname{Div}}
\newcommand{\tr}{\operatorname{tr}}
\newcommand{\la}{\langle}
\newcommand{\ra}{\rangle}

\title{Local Propagation of Momentum Moments for the Relativistic Vlasov--Poisson System}
\author{Zhaopeng Wang\\
School of Mathematics and Statistics\\ Beijing Institute of Technology\\
Beijing 100081, China\\
\texttt{zpwang@bit.edu.cn}}
\date{}

\begin{document}

\maketitle

\begin{abstract}
This paper proves local propagation of momentum moments of order
$k\ge5/2$ for the three-dimensional relativistic Vlasov--Poisson
system. In the Coulomb case, this improves the previous result which
requires integer moments of order $k\ge3$ \cite{LeopoldSaffirio}. The proof adapts the
critical density method developed in our recent work on the classical
Vlasov--Poisson equation \cite{WangCritical}. We introduce a weighted
kinetic tensor adapted to the relativistic velocity and obtain the
critical density estimate. This estimate then allows us to close the
moment argument by a short-time bootstrap.
\end{abstract}



\section{Introduction}

The Vlasov equation describes a large collection of collisionless particles moving under a self-consistent force. In a plasma, the force is electrostatic. In a stellar model, it is attractive. When the particle speed is not close to the speed of light, the classical Vlasov--Poisson system is a natural model. When relativistic effects are kept, the transport velocity has to be changed.

In this paper we consider the three-dimensional relativistic Vlasov--Poisson system in the whole space. We use units in which the speed of light is one. The unknown $f=f(t,x,p)\ge0$ is the particle density in phase space, where $x\in\R^3$ is the position and $p\in\R^3$ is the momentum. The Cauchy problem is
\begin{equation}
\left\{
\begin{aligned}
 &\partial_t f+\widehat p\cdot\nabla_x f+\gamma E\cdot\nabla_p f=0,
 \qquad f|_{t=0}=f_0,\\
 &\rho(t,x)=\int_{\R^3}f(t,x,p)\dd p,
 \qquad E=-\nabla U,\qquad -\Delta U=\rho.
\end{aligned}
\right.
\label{eq:RVP}
\end{equation}
Here $\widehat p=p/\sqrt{1+|p|^2}$ is the relativistic velocity. The function $U=U(t,x)$ is the self-consistent potential, $E=E(t,x)$ is the self-consistent electric field, and $\rho=\rho(t,x)$ is the spatial density. We write $\la p\ra=\sqrt{1+|p|^2}$ and normalize $U$ by $U(t,x)\to0$ as $|x|\to\infty$. The choice $\gamma=1$ gives the repulsive plasma case, while $\gamma=-1$ gives the attractive case. The relativistic total energy is
\[
 \mathcal E(t)=\iint_{\R^6}\la p\ra f(t,x,p)\dd x\dd p
 +\frac{\gamma}{2}\int_{\R^3}|E(t,x)|^2\dd x.
\]

The system \eqref{eq:RVP} is the relativistic counterpart of the classical Vlasov--Poisson equation. If the relativistic velocity $\widehat p$ is replaced by the momentum $p$, one recovers the classical Vlasov--Poisson equation. This change is important for moments. In the classical model, the kinetic energy is a second velocity moment. In the relativistic model, the particle-energy term is
\[
 \iint_{\R^6}\la p\ra f\dd x\dd p,
\]
which is exactly the first momentum moment in the notation used below. Thus the kinetic part of the relativistic energy has one fewer power of the phase variable than in the classical problem.

The classical Vlasov--Poisson theory is much more complete. In three dimensions, global classical solutions for general smooth data were obtained by Lions and Perthame \cite{LionsPerthame}, Pfaffelmoser \cite{Pfaffelmoser}, and Schaeffer \cite{Schaeffer}. In the classical problem the phase variable $v$ is the particle velocity, and the $k$th velocity moment is
\[
 \iint_{\R^6}|v|^k f(t,x,v)\dd x\dd v.
\]
Lions and Perthame \cite{LionsPerthame} propagated moments above order three. Pallard \cite{Pallard} later lowered the whole-space threshold to every order strictly larger than two. Moment propagation has also been studied for several related Vlasov--Poisson models, including models with massless electrons \cite{GriffinPickeringIacobelliCPDE,GriffinPickeringIacobelliKRM}, magnetized Vlasov--Poisson equations \cite{BenPoratGagnebinIacobelliJunne}, and the plasma--charge model \cite{WangPlasmaCharge}.

In \cite{WangCritical}, we introduced a different approach to low-order moment propagation. A positive kinetic tensor and Serre's compensated-integrability theorem \cite{Serre2018,Serre2019} give the critical bounds
\[
 \rho\in L^2_{t,x},
 \qquad
 E\in L^2_tL^6_x.
\]
These bounds provide the key estimates for the moment argument. The present paper adapts this mechanism to the relativistic case.

The relativistic theory is less complete. Glassey and Schaeffer \cite{GlasseySchaeffer1985} proved global classical solutions for spherically symmetric repulsive data. They later treated cylindrically symmetric plasma data under an angular-momentum condition \cite{GlasseySchaeffer2001}. More recently, Wang \cite{WangJSP} obtained global solutions for a class of large radial and cylindrical plasma data without compact momentum support. For general large data, however, global classical existence for the
three-dimensional repulsive relativistic Vlasov--Poisson system remains open. For the attractive problem, Glassey and Schaeffer \cite{GlasseySchaeffer1985} proved finite-time blow-up with negative energy, while Lemou, M\'ehats and Rapha\"el \cite{LemouMehatsRaphael} later constructed stable self-similar blow-up solutions. Had\v{z}i\'c and Rein \cite{HadzicRein} studied global existence and stability near steady
states, and Kiessling and Tahvildar-Zadeh \cite{KiesslingTahvildar}
obtained global-existence and blow-up criteria.

Moment propagation is also less understood in the relativistic setting. Here $p$ is the momentum, while the physical velocity is $\widehat p$. Since $|\widehat p|<1$, moments of the physical velocity are bounded by the mass and carry no high-momentum information. The nontrivial quantities are the momentum moments. We set
\begin{equation}\label{eq:m}
 m_k(t,x)=\int_{\R^3}\la p\ra^k f(t,x,p)\dd p,
 \qquad
 M_k(t)=\int_{\R^3}m_k(t,x)\dd x.
\end{equation}
In particular, the particle-energy term above is $M_1(t)$. Since the mass is finite, finiteness of $M_k(t)$ is equivalent to finiteness of the corresponding $|p|^k$ moment. Wang \cite{WangJSP} pointed out that whether the analogue of Lions-Perthame \cite{LionsPerthame} holds for the three-dimensional relativistic Vlasov--Poisson system remains an open problem. The present paper addresses one aspect of this question by establishing a local-in-time propagation result for low-order momentum moments.

Leopold and Saffirio \cite{LeopoldSaffirio} proved local-in-time propagation of momentum moments for relativistic Vlasov equations with singular interactions. In the three-dimensional Coulomb case, their result gives local-in-time propagation for integer moments of order $k\ge 3$. This paper lowers the threshold to $k\geq5/2$. The main new ingredient is a weighted kinetic tensor adapted to the relativistic velocity. The weight is chosen so that the tensor satisfies the same critical determinant bound as in the classical case.

Our main result is the following.

\begin{theorem}\label{thm:main}
For either choice of $\gamma$, assume that $0\le f_0\in L^1(\R^6)\cap L^\infty(\R^6)$ and
\[
 M_k(0)<\infty
 \qquad\text{for some }k\ge\frac52.
\]
Then there exist $T>0$ and a weak solution $0\leq f\in L^\infty(0,T;L^1(\R^6)\cap L^\infty(\R^6))$  of \eqref{eq:RVP} such that

\begin{equation}
 \sup_{0\le t\le T}M_k(t)<\infty.
\label{eq:main-moment}
\end{equation}
The time $T$ and the bound in \eqref{eq:main-moment} depend only on $k$, $\|f_0\|_1$, $\|f_0\|_\infty$, and $M_k(0)$.  
\end{theorem}

\begin{remark}
The threshold $k\ge5/2$ is tied to the relativistic tensor introduced
in this paper and may not be optimal. Our main purpose is to develop
this tensor structure and show that the critical density method can be
adapted to the relativistic setting. A different tensor or a sharper
field estimate may lead to a lower threshold. We do not pursue this
question here.
\end{remark}
\begin{remark}
The weak solution in Theorem \ref{thm:main} may also be chosen to conserve the relativistic total energy $\mathcal E(t)$ on $[0,T]$. This follows from the critical density estimate below and the energy-conservation argument in \cite{WangCritical}.
\end{remark}
\begin{remark}
For classical solutions, Proposition~\ref{prop:density} also gives the continuation criterion
\[
\sup_{0\le t<T_{\max}}M_{5/2}(t)<\infty
\quad\Longrightarrow\quad
T_{\max}=\infty.
\]
Indeed, the density bound yields $\rho\in L^2((0,T_{\max})\times\mathbb R^3)$ and the standard characteristic argument \cite{Rein} gives bounded momentum support.
\end{remark}

The main point is to find a relativistic version of the critical density argument in \cite{WangCritical}. For this purpose, we introduce the weighted tensor
\[
A(t,x)=\int_{\R^3}\la p\ra^{5/2}
\binom{1}{\widehat p}\otimes
\binom{1}{\widehat p}
f(t,x,p)\dd p.
\]
Although this tensor is different from the one used for the classical Vlasov--Poisson equation, it satisfies the same critical determinant estimate,
\[
\det A\ge c\|f_0\|_\infty^{-2}\rho^6.
\]
Serre's compensated-integrability estimate \cite{Serre2018,Serre2019} then reduces the density estimate to the control of the divergence of $A$. The new divergence term is bounded by $|E|m_{3/2}$ with $m_k$ as defined in \eqref{eq:m}. Kinetic interpolation gives the corresponding $L^{6/5}_x$ bound for
$m_{3/2}$, while the Poisson equation controls $E$ by $\rho$. These
estimates close the Serre inequality and show that control of the \(5/2\)-moment yields
\[
\rho\in L^2_{t,x},\quad E\in L^2_t L^6_x.
\]
Hence the critical density argument can still be closed in the
relativistic case. The rest of the proof combines this critical density estimate with standard kinetic interpolation. For every $k\ge5/2$, interpolation first gives
\[
 M_k
 \longrightarrow
 M_{5/2}
 \longrightarrow
 \rho\in L^2_{t,x}.
\]
The same $k$th moment also gives
\[
 \rho\in L^\infty_tL^{(k+3)/3}_x.
\]
Interpolating this bound with the critical $L^2_{t,x}$ estimate and
using the Poisson equation yields
\[
 E\in L^3_tL^{k+3}_x.
\]
Therefore the estimates close the cycle
\[
 M_k
 \xrightarrow{\mathrm{interpolation}}
 M_{5/2}
 \xrightarrow{\mathrm{Prop.~\ref{prop:density}}}
 \rho\in L^2_{t,x}
 \xrightarrow[\text{with }M_k]{\mathrm{Lem.~\ref{lem:field-k}}}
 E\in L^3_tL^{k+3}_x
 \xrightarrow{\mathrm{Prop.~\ref{prop:moment}}}
 M_k.
\]
A short-time bootstrap then gives the desired propagation of $M_k$.

The rest of the paper is organized as follows. Section 2 proves the critical density estimate and Section 3 proves the propagation of momentum moments. We work with smooth regularizations and establish estimates that are uniform in the regularization parameter. The passage to a weak solution then follows from the standard compactness argument, which we do not repeat here.

\section{A critical density estimate}

Throughout this section, we work with a smooth solution corresponding
to a nonnegative smooth compactly supported approximation of the
initial datum. We suppress the approximation parameter. The
approximations are chosen so that
\[
 \|f(0)\|_1\le \|f_0\|_1,
 \qquad
 \|f(0)\|_\infty\le \|f_0\|_\infty.
\]
For each fixed approximation, all the quantities used below are finite
on the time interval under consideration. The estimates obtained in this section
are independent of the approximation parameter and of the size of the
support.

\subsection{The relativistic tensor}

We use Serre's Euclidean compensated-integrability theorem
\cite{Serre2018,Serre2019}. We state only the four-dimensional form
needed here.

\begin{theorem}\label{thm:serre}

Let $B\in L^1(\R^4;\mathrm{Sym}_4^+)$ and assume that $\Div B$ is a
finite $\R^4$-valued Radon measure. Then
\[
 \int_{\R^4}(\det B)^{1/3}\dd t\dd x
 \le
 C\|\Div B\|_{\mathcal M(\R^4)}^{4/3}.
\]

\end{theorem}

We introduce a relativistic kinetic tensor adapted to the velocity
$\widehat p$. It is different from the tensor used in
\cite{WangCritical}. Define
\begin{equation}
 z(p)=\binom{1}{\widehat p},
 \qquad
 A(t,x)
 =
 \int_{\R^3}
 \la p\ra^{5/2}z(p)\otimes z(p)f(t,x,p)\dd p.
\label{eq:tensor}
\end{equation}
The tensor $A$ is symmetric and positive semidefinite. Indeed, for every
$\xi\in\R^4$,
\[
 \xi^TA\xi
 =
 \int_{\R^3}
 \la p\ra^{5/2}
 (\xi\cdot z(p))^2
 f(t,x,p)\dd p
 \ge0.
\]

The main point is the following determinant estimate.

\begin{lemma}\label{lem:det}

Let $0\le g\in L^1(\R^3)\cap L^\infty(\R^3)$ and assume that
\[
 \int_{\R^3}\la p\ra^{5/2}g(p)\dd p<\infty.
\]
Set
\[
 \rho_g=\int_{\R^3}g(p)\dd p,
 \qquad
 A_g
 =
 \int_{\R^3}
 \la p\ra^{5/2}
 z(p)\otimes z(p)g(p)\dd p.
\]
Then
\begin{equation}
 \det A_g
 \ge
 c\|g\|_\infty^{-2}\rho_g^6.\nonumber
\end{equation}
\end{lemma}

\begin{proof}
There is nothing to prove if $\rho_g=0$. Assume $\rho_g>0$ and write
\[
 w(p):=\la p\ra^{5/2},
 \qquad
 M:=\int_{\R^3}w(p)g(p)\dd p.
\]
We use the relativistic velocity
\[
 q=\widehat p=\frac{p}{\la p\ra}.
\]
Define its weighted mean and covariance by
\[
 \bar q
 :=
 \frac1M\int_{\R^3}q\,w(p)g(p)\dd p,
\]
and
\[
 \mathcal V_g
 :=
 \frac1M\int_{\R^3}
 (q-\bar q)\otimes(q-\bar q)\,w(p)g(p)\dd p.
\]
The matrix $\mathcal V_g$ is positive definite. Indeed, if
$\xi^T\mathcal V_g\xi=0$ for some $\xi\ne0$, then
\[
 \xi\cdot(q-\bar q)=0
\]
for almost every $p$ with $g(p)>0$. The map $p\mapsto\widehat p$ is a
smooth diffeomorphism from $\R^3$ onto $B_1$, while an affine plane in
$B_1$ has three-dimensional measure zero. This is impossible because
$\rho_g>0$.

Since
\[
 \int_{\R^3}(q-\bar q)w(p)g(p)\dd p=0,
\]
we have
\[
 \int_{\R^3}q\otimes q\,w(p)g(p)\dd p
 =M\bar q\otimes\bar q+M\mathcal V_g.
\]
Hence
\[
 A_g
 =
 \begin{pmatrix}
 M & M\bar q^T\\
 M\bar q & M\bar q\otimes\bar q+M\mathcal V_g
 \end{pmatrix}.
\]
Taking the Schur complement of the upper-left entry gives
\begin{equation}\label{eq:rvp-det-covariance}
 \det A_g=M^4\det\mathcal V_g.
\end{equation}

We now normalise the covariance. Set
\[
 X(p)
 :=
 (q-\bar q)^T\mathcal V_g^{-1}(q-\bar q).
\]
By the definition of $\mathcal V_g$,
\[
 \begin{aligned}
 \int_{\R^3}X(p)w(p)g(p)\dd p
 =
 \tr\left(
 \mathcal V_g^{-1}
 \int_{\R^3}(q-\bar q)\otimes(q-\bar q)w(p)g(p)\dd p
 \right)=3M.
 \end{aligned}
\]
Therefore
\begin{equation}\label{eq:rvp-normalised-average}
 \int_{\R^3}
 \frac{w(p)}M\bigl(1+X(p)\bigr)g(p)\dd p=4.
\end{equation}
Set
\[
 R:=\frac8{\rho_g}
\]
and define
\[
 \Omega
 :=
 \left\{
 p\in\R^3:
 \frac{w(p)}M\bigl(1+X(p)\bigr)\le R
 \right\}.
\]
By \eqref{eq:rvp-normalised-average} and Markov's inequality,
\[
 \int_{\Omega^c}g(p)\dd p
 \le
 \frac4R
 =
 \frac{\rho_g}{2}.
\]
Thus at least half of the mass lies in $\Omega$:
\[
 \int_\Omega g(p)\dd p\ge\frac{\rho_g}{2}.
\]
Since $g\le\|g\|_\infty$,
\begin{equation}\label{eq:omega-lower}
 |\Omega|
 \ge
 \frac{\rho_g}{2\|g\|_\infty}.
\end{equation}

It remains to estimate the volume of $\Omega$ from above. On $\Omega$,
\[
 1
 \le
 R^2M^2w(p)^{-2}\bigl(1+X(p)\bigr)^{-2}.
\]
Hence
\begin{equation}\label{eq:rvp-omega-upper-1}
 |\Omega|
 \le
 R^2M^2
 \int_{\R^3}
 \la p\ra^{-5}\bigl(1+X(p)\bigr)^{-2}\dd p.
\end{equation}
Here the exponent $5/2$ is exactly the one needed for the relativistic
change of variables. Indeed,
\[
 q=\widehat p=\frac{p}{\la p\ra}
\]
maps $\R^3$ onto $B_1$ and
\[
 D_p\widehat p
 =
 \la p\ra^{-1}I-\la p\ra^{-3}p\otimes p.
\]
It has two tangential eigenvalues $\la p\ra^{-1}$ and one radial
eigenvalue $\la p\ra^{-3}$. Thus
\begin{equation}\label{eq:rel-jacobian}
 \dd q=\la p\ra^{-5}\dd p.
\end{equation}
The square of the weight $w(p)=\la p\ra^{5/2}$ therefore produces
exactly the Jacobian in \eqref{eq:rel-jacobian}.

Using $q=\widehat p$ in \eqref{eq:rvp-omega-upper-1},
\[
 \begin{aligned}
 \int_{\R^3}
 \la p\ra^{-5}\bigl(1+X(p)\bigr)^{-2}\dd p
 &=
 \int_{B_1}
 \left[
 1+(q-\bar q)^T\mathcal V_g^{-1}(q-\bar q)
 \right]^{-2}\dd q.
 \end{aligned}
\]
Now set
\[
 y=\mathcal V_g^{-1/2}(q-\bar q).
\]
Since
\[
 \dd q=(\det\mathcal V_g)^{1/2}\dd y,
\]
we obtain
\[
 \begin{aligned}
 \int_{B_1}
 \left[
 1+(q-\bar q)^T\mathcal V_g^{-1}(q-\bar q)
 \right]^{-2}\dd q
 &\le
 (\det\mathcal V_g)^{1/2}
 \int_{\R^3}(1+|y|^2)^{-2}\dd y\\
 &\le
 C(\det\mathcal V_g)^{1/2}.
 \end{aligned}
\]
Together with \eqref{eq:rvp-det-covariance}, this gives
\[
 M^2(\det\mathcal V_g)^{1/2}=\sqrt{\det A_g},
\]
and therefore
\[
 |\Omega|
 \le
 CR^2\sqrt{\det A_g}
 \le
 C\rho_g^{-2}\sqrt{\det A_g}.
\]
Combining this with \eqref{eq:omega-lower},
\[
 \frac{\rho_g}{2\|g\|_\infty}
 \le
 C\rho_g^{-2}\sqrt{\det A_g}.
\]
Hence
\[
 \det A_g
 \ge
 c\|g\|_\infty^{-2}\rho_g^6.
\]
This completes the proof.
\end{proof}

\subsection{The density bound}

Lemma \ref{lem:det} has the power needed for Serre's theorem. In four
space--time dimensions, Serre controls $(\det A)^{1/3}$, while
\[
 (\det A)^{1/3}
 \ge
 c\|f_0\|_\infty^{-2/3}\rho^2.
\]
It remains to estimate the divergence of $A$.

For $\beta=0,1,2,3$, we use the row-wise convention
\[
 (\Div A)_\beta
 =
 \partial_tA_{\beta0}
 +
 \sum_{j=1}^3\partial_{x_j}A_{\beta j}.
\]
Since $z_0=1$ and $(z_1,z_2,z_3)=\widehat p$, the definition of $A$ gives
\[
 (\Div A)_\beta
 =
 \int_{\R^3}
 \la p\ra^{5/2}z_\beta(p)
 \bigl(\partial_tf+\widehat p\cdot\nabla_xf\bigr)\dd p.
\]
Using \eqref{eq:RVP} and integrating by parts in $p$,
\[
 (\Div A)_\beta
 =
 \gamma
 \int_{\R^3}
 f(t,x,p)E(t,x)\cdot
 \nabla_p\bigl(\la p\ra^{5/2}z_\beta(p)\bigr)\dd p.
\]
For $\beta=0$,
\[
 \big|\nabla_p\la p\ra^{5/2}\big|
 \le C\la p\ra^{3/2}.
\]
For $\beta=1,2,3$,
\[
 \la p\ra^{5/2}z_\beta(p)=\la p\ra^{3/2}p_\beta,
\]
and therefore
\[
 \big|\nabla_p(\la p\ra^{3/2}p_\beta)\big|
 \le C\la p\ra^{3/2}.
\]
Thus
\begin{equation}\label{eq:divA}
 |\Div A(t,x)|
 \le
 C|E(t,x)|m_{3/2}(t,x).
\end{equation}

The factor $m_{3/2}$ will be paired with $E\in L^6_x$. We therefore
need an $L^{6/5}_x$ estimate. We first record the standard kinetic
interpolation inequality used here and again in Section 3.

\begin{lemma}\label{lem:kinetic-interpolation}
Let $k>0$ and $0\le\ell\le k$. Then, for every $t$,
\[
 \|m_\ell(t)\|_{(k+3)/(\ell+3)}
 \le
 C_{k,\ell}
 \|f_0\|_\infty^{(k-\ell)/(k+3)}
 M_k(t)^{(\ell+3)/(k+3)}.
\]
\end{lemma}

\begin{proof}
The cases $\|f_0\|_\infty=0$ and $\ell=k$ are immediate. Assume
$0\le\ell<k$ and fix $x$. The Vlasov flow preserves the $L^\infty$
norm, hence
\[
 \|f(t)\|_\infty\le\|f_0\|_\infty.
\]
For $R\ge1$,
\[
 \begin{aligned}
 m_\ell(t,x)
 &\le
 \int_{|p|\le R}\la p\ra^\ell f\dd p
 +
 \int_{|p|>R}\la p\ra^\ell f\dd p\\
 &\le
 C\|f_0\|_\infty R^{\ell+3}
 +
 CR^{\ell-k}m_k(t,x).
 \end{aligned}
\]
If $m_k(t,x)\le\|f_0\|_\infty$, then
\[
 m_\ell(t,x)
 \le
 m_k(t,x)
 \le
 \|f_0\|_\infty^{(k-\ell)/(k+3)}
 m_k(t,x)^{(\ell+3)/(k+3)}.
\]
If $m_k(t,x)>\|f_0\|_\infty$, take
\[
 R=
 \left(
 \frac{m_k(t,x)}{\|f_0\|_\infty}
 \right)^{1/(k+3)}.
\]
The two terms then have the same size, and
\[
 m_\ell(t,x)
 \le
 C_{k,\ell}
 \|f_0\|_\infty^{(k-\ell)/(k+3)}
 m_k(t,x)^{(\ell+3)/(k+3)}.
\]
Raise this inequality to the power $(k+3)/(\ell+3)$ and integrate in
$x$.
\end{proof}

With $k=5/2$ and $\ell=3/2$,
\[
 \|m_{3/2}(t)\|_{11/9}
 \le
 C\|f_0\|_\infty^{2/11}
 M_{5/2}(t)^{9/11}.
\]
We also have, by H\"older's inequality with respect to
$f(t,x,p)\dd x\dd p$ and conservation of mass,
\[
 \|m_{3/2}(t)\|_1
 =M_{3/2}(t)
 \le
 \|f_0\|_1^{2/5}M_{5/2}(t)^{3/5}.
\]
Since
\[
 \frac56
 =
 \frac1{12}
 +
 \frac{11}{12}\frac9{11},
\]
interpolation between $L^1_x$ and $L^{11/9}_x$ gives
\begin{equation}\label{eq:m32}
 \|m_{3/2}(t)\|_{6/5}
 \le
 C
 \|f_0\|_\infty^{1/6}
 \|f_0\|_1^{1/30}
 M_{5/2}(t)^{4/5}.
\end{equation}

\begin{proposition}\label{prop:density}
Let $f$ be a smooth solution on $[0,T]$ corresponding to one of the
regularized initial data described above. Then
\[
 \|\rho\|_{L^2((0,T)\times\R^3)}
 \le
 C
 \left(
 \sup_{0\le t\le T}M_{5/2}(t)
 \right)^{2/3}
 +
 CT
 \left(
 \sup_{0\le t\le T}M_{5/2}(t)
 \right)^{8/5}.
\]
Moreover,
\begin{equation}\label{eq:fieldbound}
 \|E\|_{L^2(0,T;L^6(\R^3))}
 \le
 C\|\rho\|_{L^2((0,T)\times\R^3)}.
\end{equation}
Here $C$ depends only on $\|f_0\|_1$ and $\|f_0\|_\infty$. In
particular, it is independent of the regularization, the support of the
regularized solution, and the sign of $\gamma$.
\end{proposition}

\begin{proof}
Put
\[
 K_T:=\sup_{0\le t\le T}M_{5/2}(t).
\]
Extend $A$ by zero outside $(0,T)$. Since $|z(p)|\le C$,
\[
 |A(t,x)|
 \le
 C m_{5/2}(t,x),
\]
and hence
\[
 \int_{\R^4}|A(t,x)|\dd t\dd x
 \le
 CTK_T.
\]
The zero extension is therefore in $L^1(\R^4)$.

Its divergence has an interior part and two time-face measures. By
\eqref{eq:divA}, the interior part is bounded by $C|E|m_{3/2}$. For the
smooth regularized solution this is integrable. On $t=0$ and $t=T$,
\[
 |A_{\beta0}(t,x)|\le Cm_{5/2}(t,x),
 \qquad \beta=0,1,2,3,
\]
so the two boundary measures have total mass at most $CK_T$. Thus
\[
 \|\Div A\|_{\mathcal M(\R^4)}
 \le
 C\left(
 K_T+
 \int_0^T\int_{\R^3}|E|m_{3/2}\dd x\dd t
 \right).
\]
Theorem \ref{thm:serre} and Lemma \ref{lem:det} give
\begin{equation}\label{eq:serre-rho}
 \|\rho\|_{L^2((0,T)\times\R^3)}
 \le
 C\|f_0\|_\infty^{1/3}
 \left(
 K_T+
 \int_0^T\int_{\R^3}|E|m_{3/2}\dd x\dd t
 \right)^{2/3}.
\end{equation}
The Poisson equation and the Hardy--Littlewood--Sobolev inequality give
\[
 \|E(t)\|_6\le C\|\rho(t)\|_2.
\]
After integration in time,
\[
 \|E\|_{L^2(0,T;L^6)}
 \le
 C\|\rho\|_{L^2((0,T)\times\R^3)}.
\]
Using \eqref{eq:m32} and \eqref{eq:fieldbound},
\[
 \begin{aligned}
 \int_0^T\int_{\R^3}|E|m_{3/2}\dd x\dd t
 &\le
 \int_0^T\|E(t)\|_6\|m_{3/2}(t)\|_{6/5}\dd t\\
 &\le
 CT^{1/2}
 \|f_0\|_\infty^{1/6}
 \|f_0\|_1^{1/30}
 K_T^{4/5}
 \|\rho\|_{L^2((0,T)\times\R^3)}.
 \end{aligned}
\]
Insert this into \eqref{eq:serre-rho}. Since
$(a+b)^{2/3}\le a^{2/3}+b^{2/3}$,
\[
 \begin{aligned}
 \|\rho\|_{L^2((0,T)\times\R^3)}
 \le {}
 C\|f_0\|_\infty^{1/3}K_T^{2/3}
 +
 CT^{1/3}
 \|f_0\|_\infty^{4/9}
 \|f_0\|_1^{1/45}
 K_T^{8/15}
 \|\rho\|_{L^2((0,T)\times\R^3)}^{2/3}.
 \end{aligned}
\]
Young's inequality gives
\[
 \begin{aligned}
 CT^{1/3}
 \|f_0\|_\infty^{4/9}
 \|f_0\|_1^{1/45}
 K_T^{8/15}
 \|\rho\|_{L^2((0,T)\times\R^3)}^{2/3}\le
 \frac12\|\rho\|_{L^2((0,T)\times\R^3)}
 +
 CT\|f_0\|_\infty^{4/3}
 \|f_0\|_1^{1/15}K_T^{8/5}.
 \end{aligned}
\]
Absorbing the first term gives the stated bound.
\end{proof}
 
\section{Propagation of moments}

We keep the smooth regularized solutions from Section 2. The first step is to place the field in the space that is dual to the
local $(k-1)$st moment.

\begin{lemma}\label{lem:field-k}
Let $k\ge5/2$. Then
\begin{equation}\label{eq:field-k}
 \|E\|_{L^3(0,T;L^{k+3}(\R^3))}
 \le
 C_k\|\rho\|_{L^2((0,T)\times\R^3)}^{2/3}
 \|f_0\|_\infty^{k/[3(k+3)]}
 \left(
 \sup_{0\le t\le T}M_k(t)
 \right)^{1/(k+3)}.
\end{equation}
\end{lemma}

\begin{proof}
Lemma \ref{lem:kinetic-interpolation} with $\ell=0$ gives
\[
 \|\rho(t)\|_{(k+3)/3}
 \le
 C_k\|f_0\|_\infty^{k/(k+3)}
 M_k(t)^{3/(k+3)}.
\]
Choose
\[
 r_k:=\frac{3(k+3)}{k+6}.
\]
This exponent is chosen so that the Riesz potential of order one maps
$L^{r_k}$ to $L^{k+3}$:
\[
 \frac1{k+3}=\frac1{r_k}-\frac13.
\]
Also,
\[
 \frac1{r_k}
 =
 \frac23\frac12
 +
 \frac13\frac3{k+3}.
\]
Hence
\[
 \|\rho(t)\|_{r_k}
 \le
 \|\rho(t)\|_2^{2/3}
 \|\rho(t)\|_{(k+3)/3}^{1/3}.
\]
The Hardy--Littlewood--Sobolev inequality therefore gives
\[
 \|E(t)\|_{k+3}
 \le
 C_k
 \|f_0\|_\infty^{k/[3(k+3)]}
 M_k(t)^{1/(k+3)}
 \|\rho(t)\|_2^{2/3}.
\]
Taking the $L^3$ norm in time gives \eqref{eq:field-k}, since
$(\|\rho(t)\|_2^{2/3})^3=\|\rho(t)\|_2^2$.
\end{proof}

We next use the standard differential inequality for $M_k$. Since
\[
 \nabla_p\la p\ra^k
 =
 k\la p\ra^{k-2}p,
\]
the Vlasov equation gives
\[
 \frac{\dd}{\dd t}M_k(t)
 =
 \gamma k
 \iint
 f(t,x,p)E(t,x)\cdot p\,
 \la p\ra^{k-2}\dd p\dd x.
\]
Thus
\[
 \left|\frac{\dd}{\dd t}M_k(t)\right|
 \le
 C_k\int_{\R^3}|E(t,x)|m_{k-1}(t,x)\dd x.
\]
Lemma \ref{lem:kinetic-interpolation} with $\ell=k-1$ gives
\[
 \|m_{k-1}(t)\|_{(k+3)/(k+2)}
 \le
 C_k\|f_0\|_\infty^{1/(k+3)}
 M_k(t)^{(k+2)/(k+3)}.
\]
The exponent $(k+3)/(k+2)$ is dual to $k+3$. Hence
\begin{equation} 
 \left|\frac{\dd}{\dd t}M_k(t)\right|
 \le
 C_k\|f_0\|_\infty^{1/(k+3)}
 M_k(t)^{(k+2)/(k+3)}
 \|E(t)\|_{k+3}.\nonumber
\end{equation}
The power of $M_k$ is removed by taking the $(k+3)$rd root. We obtain
\begin{equation}\label{eq:moment-root}
 (1+M_k(t))^{1/(k+3)}
 \le
 (1+M_k(0))^{1/(k+3)}
 +
 C_k\|f_0\|_\infty^{1/(k+3)}
 \int_0^t\|E(s)\|_{k+3}\dd s.
\end{equation}

\begin{proposition}\label{prop:moment}
Let $k\ge5/2$. There exist $T_*>0$ and $C>0$, depending only on
\[
 k,\qquad
 \|f_0\|_1,\qquad
 \|f_0\|_\infty,\qquad
 M_k(0),
\]
such that every smooth regularized solution satisfies
\[
 \sup_{0\le t\le T_*}M_k(t)\le C.
\]
Both $T_*$ and $C$ are independent of the regularization.
\end{proposition}

\begin{proof}
For $T>0$, put
\[
 \mathcal Y(T)
 :=
 \sup_{0\le t\le T}(1+M_k(t))^{1/(k+3)},
 \qquad
 Y_0:=(1+M_k(0))^{1/(k+3)}.
\]
The assumption $k\ge5/2$ enters here. By interpolation of momentum
moments,
\[
 M_{5/2}(t)
 \le
 \|f_0\|_1^{1-5/(2k)}M_k(t)^{5/(2k)}.
\]
Therefore
\begin{equation} 
 \sup_{0\le t\le T}M_{5/2}(t)
 \le
 C(\|f_0\|_1)
 \mathcal Y(T)^{5(k+3)/(2k)}.\nonumber
\end{equation}
Proposition \ref{prop:density} then gives
\begin{equation}\label{eq:rho-Y}
 \|\rho\|_{L^2((0,T)\times\R^3)}
 \le
 C\mathcal Y(T)^{5(k+3)/(3k)}
 +
 CT\mathcal Y(T)^{4(k+3)/k}.
\end{equation}
Here and below $C$ depends only on
$k$, $\|f_0\|_1$, and $\|f_0\|_\infty$.

Lemma \ref{lem:field-k} gives
\[
 \|E\|_{L^3(0,T;L^{k+3})}
 \le
 C\mathcal Y(T)
 \|\rho\|_{L^2((0,T)\times\R^3)}^{2/3}.
\]
Using \eqref{eq:moment-root} and H\"older's inequality in time,
\[
 \begin{aligned}
 \mathcal Y(T)
 &\le
 Y_0+C\int_0^T\|E(t)\|_{k+3}\dd t\\
 &\le
 Y_0+CT^{2/3}
 \|E\|_{L^3(0,T;L^{k+3})}\\
 &\le
 Y_0+CT^{2/3}\mathcal Y(T)
 \|\rho\|_{L^2((0,T)\times\R^3)}^{2/3}.
 \end{aligned}
\]
Insert \eqref{eq:rho-Y}. Since
$(a+b)^{2/3}\le a^{2/3}+b^{2/3}$,
\begin{equation}\label{eq:bootstrap}
 \mathcal Y(T)
 \le
 Y_0
 +
 C_1T^{2/3}
 \mathcal Y(T)^{1+10(k+3)/(9k)}
 +
 C_2T^{4/3}
 \mathcal Y(T)^{1+8(k+3)/(3k)}.
\end{equation}
Choose $T_*>0$ so small that
\[
 C_1T_*^{2/3}(2Y_0)^{1+10(k+3)/(9k)}
 +
 C_2T_*^{4/3}(2Y_0)^{1+8(k+3)/(3k)}
 \le
 \frac{Y_0}{2}.
\]
For the smooth solution, $M_k(t)$ is continuous. If
$\mathcal Y(T)$ reached $2Y_0$ for the first time at some
$\tau\le T_*$, then \eqref{eq:bootstrap} with $T=\tau$ would give
\[
 2Y_0\le Y_0+\frac{Y_0}{2},
\]
a contradiction. Hence
\[
 \mathcal Y(T)\le2Y_0
 \qquad (0\le T\le T_*),
\]
and therefore
\[
 \sup_{0\le t\le T_*}M_k(t)
 \le
 (2Y_0)^{k+3}.
\]
This completes the proof.
\end{proof}

The estimate above is uniform in the regularization parameter. Together
with the bounds from Section 2, it gives the uniform a priori control
needed in Theorem \ref{thm:main}. The remaining compactness step is the
standard one used for the regularized Vlasov--Poisson system \cite{ChenChen, Pallard} and is therefore omitted.

\paragraph{Data/Code Availability.}
The manuscript has no associated data or code.

\paragraph{Conflict of interest.}
The author declares that he has no conflict of interest.

\end{document}